\documentclass[12pt,a4paper]{article}
\usepackage[top=3cm, bottom=3cm, left=3cm, right=3cm]{geometry}
\usepackage[latin1]{inputenc}
\usepackage{amsmath}
\usepackage{amsthm}
\usepackage{amsfonts}
\usepackage{amssymb}
\usepackage{graphics}
\usepackage[hang,flushmargin]{footmisc} 

\usepackage{amsmath,amsthm,amssymb,graphicx, multicol, array}
\usepackage{enumerate}
\usepackage{enumitem}
\usepackage{xcolor}
\usepackage{currfile}

\usepackage[bookmarks,colorlinks,breaklinks, pagebackref]{hyperref}

\hypersetup{linkcolor=blue,citecolor=red,filecolor=blue,urlcolor=blue} 

\usepackage{tabto}

\usepackage{tikz}
\usepackage{pgfplots}

\DeclareMathOperator{\li}{li}

\newtheorem{thm}{Theorem}[section]

\newtheorem{lem}{Lemma}[section]

\newtheorem{conj}{Conjecture}[section]
\newtheorem{exa}{Example}[section]

\newtheorem{dfn}{Definition}[section]

\newtheorem{rmk}{Remark}[section]

\newcommand{\Q}{\mathbb{Q}}
\newcommand{\R}{\mathbb{R}}
\newcommand{\C}{\mathbb{C}}

\usepackage{fancyhdr}
\title{Note On The Catalan Constant And Prime Triples}
\date{}
\author{N. A Carella}

\begin{document}
\maketitle
\begin{abstract}
The existence of infinitely many consecutive prime triples $p_n$, $ p_{n+1}$, and $p_{n+2}$ as $n \to \infty$, is sufficient to prove that the Catalan constant $\beta(2)=0.9159655941\ldots $ is an irrational number. This note provides the detailed analysis. Moreover, the numerical data suggests that the irrationality measure is $\mu(\beta(2))=2$, the same as almost every irrational real numbers.
\let\thefootnote\relax\footnote{ \today \date{} 
}
\end{abstract}



\section{Introduction and the Result} \label{S2222}

The Dirichlet beta function is defined by the series
\begin{equation}\label{eq2222.100}
\beta(s)=\sum_{n \geq 1} \frac{\chi(n)}{n^s}=\prod_{p \geq 3} \left (1-\frac{\chi(p)}{p^s}\right )^{-1},
\end{equation}
where $\chi(n)$ is the quadratic symbol, and $s \in \C$ is a complex number. A beta constant $\beta(s)$ at an odd integer argument $s=2n+1$ has an exact evaluation
\begin{equation} \label{eq2222.110}
 \beta(2n+1)=\frac{\pi^{2n+1} E_{n}}{4^{n+1}(2n)!} 
\end{equation} 
in terms of the Euler numbers 
\begin{equation} \label{eq2222.120}
E_{1}=1, \quad E_{2}=5,\quad E_{4}=1385, \quad E_{n}, \ldots, 
\end{equation}
for $n \geq 1$. This formula expresses each Dirichlet beta constant $\beta(2n+1)$ as a rational multiple of $\pi^{2n+1}$, see \cite{BB1987} and related references. In contrast, the evaluation of a beta constant at an even integer argument can involves the zeta function and a power series, and other complicated formulas, \cite{JL2017}, \cite{LF2012}, et cetera. One of the simplest of these formulas is
\begin{equation} 
\beta(s)=\frac{3}{4} \zeta(s)  -2 \sum_{n \geq 1} \frac{1}{(4n+3)^s},
\end{equation}
where $\zeta(s)$ is the zeta function and $s\geq 2$. These expressions are summarized in a compact formula.

\begin{dfn} \label{dfn2222.150}  { \normalfont Let $s \geq 2$ be an integer. The $\pi$-representation of the Dirichlet beta constant $\beta(s)$ is defined by the formula
\begin{equation} \label{eq2222.150}
\beta(s)=
\begin{cases}
\displaystyle r_n \pi^{s}&\text{if $s=2n+1$},\\
\displaystyle r_n \pi^{s}  -u_n&\text{if $s=2n$},
\end{cases}
\end{equation}
where $ r_n \in \Q$ is a rational number and $u_n \in \R$ is a real number.
}
\end{dfn}
The arithmetic nature of the first even constant, called the Catalan constant, is unknown. A proof based on a result for the existence of infinitely many consecutive prime triples 
\begin{equation}\label{eq2222.190}
 p_n=8n_1+1, \quad p_{n+1}=8n_2+ 5, \quad \text{ and }\quad p_{n+2}=8n_3+ 7
\end{equation}
as the integers variables $n_1,n_2,n_3,n\to \infty$, and other results is given here.

\begin{thm} \label{thm2222.200} The Catalan constant
\begin{equation}\label{eq2222.200}
\beta(2)=\sum_{n \geq 0} \frac{(-1)^{n}}{(2n+1)^2}=0.9159655941\ldots,\nonumber
\end{equation}
is an irrational number. 
\end{thm}

The essential basic foundation and background are presented in Section \ref{S9191}, and the simple proof of Theorem \ref{thm2222.200} is presented in Section \ref{S4577}. 

\begin{conj} \label{conj2222.300} The irrationality measure of the Catalan constant is $\mu\left (\beta(2)\right )=2.$  
\end{conj}

The experimental data compiled in Section \ref{S3300N} confirms the prediction very accurately.\\

Last, but not least, it should be observed that the same proof seamlessly generalizes to all the beta constants $\beta(s)$, where $s\geq2$ is an integer. For example, the same argument given in Section \ref{S4577} provides a new proof of the irrationality of the number $\beta(3)=\prod_{p \geq 3} \left (1-\chi(p)p^{-3}\right )^{-1}=\pi^3/32$.\\

Now, recall that an elementary argument based on the irrationality of the number $\pi^2/6=\prod_{p \geq 2} \left (1-p^{-2}\right )^{-1}$ implies the existence of infinitely many primes. Similarly, if the converse of the Brun irrationality criterion holds, see Theorem \ref{thm9191.300}, then the known irrationality of the number $\pi^3/32=\prod_{p \geq 3} \left (1-\chi(p)p^{-3}\right )^{-1}$ implies the existence of infinitely many prime triples, see \eqref{eq2222.190}. Assuming the consecutive prime triples are on three dependent arithmetic progressions, for example $n_1=n_2=n_3$ in \eqref{eq2222.190}, this is a stronger result than the twin primes conjecture. However, this is  open problem closely related to the Bateman-Horn conjecture, see \cite{FG2028} and \cite{LS2013} for more information.  
\section{Foundation}\label{S9191}
The conditional result for the Catalan constant, and more generally, the Dirichlet beta constants, is based on a proven irrationality criterion, and the conjectured distribution of prime triples. 

\subsection{Irrationality Criterion}
\begin{thm} \label{thm9191.300} {\normalfont (Brun irrationality criterion)} Let $x_n\geq 1 $ and $y_n\geq 1 $ be a pair of monotonic increasing increasing integers sequences. Suppose that the following properties are true.

\begin{enumerate}[font=\normalfont, label=(\roman*)]
\item $\displaystyle \frac{y_n}{x_n}$ converges to a real number $\alpha\ne0$ as $n\to\infty$.

\item $\displaystyle \frac{y_n}{x_n}< \frac{y_{n+1}}{x_{n+1}}$ is monotonically increasing as $n\to\infty$.

\item $\displaystyle \frac{y_{n+1}-y_n}{x_{n+1}-x_n}> \frac{y_{n+2}-y_{n+1}}{x_{n+2}-x_{n+1}}$ is monotonically decreasing as $n\to\infty$.
\end{enumerate}

Then, the number $\alpha$ is irrational.
\end{thm}

The details of the proof are discussed in \cite{BK1972}, and the most recent application of this result is given in \cite{BL2015}.

\subsection{Sequences of Consecutive Primes }
Let $p_1,p_2,p_3, \ldots, p_n, \ldots $ be the sequence of primes in increasing order. Let $q\geq1$ be an integer, and let ${\bf{a}}=(a_1,a_2,\ldots,a_k)$ be an $k$-tuple of congruences $p_n\equiv a_n \bmod q$, with $\gcd(a_n,q)=1$ for $n\in \{1,2,\ldots,k\}$. Define the primes counting function  
\begin{equation}\label{eq9191.200}
\pi(x,q,{\bf{a}})=\#\{p\leq x: p_n\equiv a_n \bmod q\}.
\end{equation}
A few results have been proved for the constant case
\begin{equation}\label{eq9191.210}
{\bf{a} }=(a_1=a,a_2=a,\ldots,a_k=a),
\end{equation}
where $\gcd(a,q)=1$, see \cite{BT2013}, \cite{LS2013}, et cetera, for the most recent literature. The nonconstant case of \eqref{eq9191.210} for $k$ dependent arithmetic progressions appears to be a difficult problem. But, for $k$ independent arithmetic progression, this problem is manageable.

\begin{thm} \label{thm9191.500} Let $x\geq 1 $ and $q\leq \log x $ be an integer. If $ {\bf{a}}=(a_1,a_2,\ldots,a_k)$ is a congruence vector such that $\gcd(a_n,q)=1$ for $n=1,2,\ldots ,k$, then, 
\begin{equation}\label{eq9191.500} 
\pi(x,q,{\bf{a}})=\frac{\li(x)}{\varphi(q)^k}\left (1+O\left ( \frac{1}{\log x}\right ) \right ).\nonumber
\end{equation}
\end{thm}
\begin{proof}[\textbf{Proof}] Let $x\geq1$ be a large number, and let $c\geq0$ be a constant. Fix a modulo $q\ll (\log x)^c$, and an admissible $k$-tuple $a_1,a_2, \ldots,a_k$ such that $\gcd(q,a_n)=1$ for $n\leq k$, take the cross product of $k$ independent arithmetic progressions
\begin{equation}\label{eq9191.510} 
qn_1+a_1, \quad qn_2+a_2, \quad \ldots, \quad qn_k+a_k.	 
\end{equation}	
By Dirichlet theorem (or Siegel-Walfisz theorem), the corresponding prime $k$-tuples has the natural density 
\begin{equation}\label{eq9191.520} 
\delta(\textbf{a})=\lim_{x \to\infty}\frac{\#\{p\leq x: p_n\equiv a_n \bmod q\}}{x}=\frac{1}{\varphi(q)^k}.	 
\end{equation}	
Each prime in the consecutive prime $k$-tuple $p_n, p_{n+1}, \ldots, p_{n+k}$ is independently generated, but satisfies the specified congruence $p_{n+i}\equiv a_{n+i} \bmod q$.
\end{proof}As a new application, the sequence of prime triples 
\begin{equation}\label{eq9191.530} 
(97,101,103),\quad(193,197, 199),\quad (457,459,463),\ldots, 
\end{equation}
defined in \eqref{eq2222.190}, is used here to develop an argument for the irrationality of the beta constant $\beta(s)$.  
\section{An Irrationality Result} \label{S4577}
The simple argument in support of Theorem \ref{thm2222.200} is the following. 

\begin{proof}[\textbf{Proof}] {\bf Theorem \ref{thm2222.200}} Suppose that 
\begin{equation}\label{eq4577.810}
 p_n\equiv 1 \bmod 8, \quad p_{n+1}\equiv 5 \bmod 8, \quad \text{ and }\quad p_{n+2}\equiv 7 \bmod 8.
\end{equation} 

This hypothesis immediately implies that
\begin{equation}\label{eq4577.820}
 \chi(p_n)= 1, \quad \chi(p_{n+1})=1, \quad \text{ and }\quad\chi(p_{n+2})=-1.
\end{equation} 
By Theorem \ref{thm9191.500}, there are infinitely many consecutive prime triples \eqref{eq4577.810} that satisfy the triple character values \eqref{eq4577.820} as $n \to \infty$. The rest of the proof verifies the three steps specified in Theorem \ref{thm9191.300} to prove the irrationality of the number $G=\beta(2)$. \\

\textbf{Condition (i): Convergence Property.} Define the sequence of rational approximations
\begin{equation}\label{eq4577.830}
\frac{y_{n}}{x_{n}}=\prod_{k \leq n} \left (1-\frac{\chi(p_k)}{p_k^2}\right )^{-1}=\prod_{k \leq n} \frac{p_k^2}{p_k^2-\chi(p_k)}.
\end{equation}
Since $\chi(p_n)=1$, the sequence $\{y_n/x_n:n\geq 1\}$ is composed of the two sequences of monotonically increasing integers
\begin{equation}\label{eq4577.840}
x_n= \prod_{k \leq n}\left (p_k^2-\chi(p_k)\right ), \quad \text{ and } \quad y_{n}=\prod_{k \leq n} p_k^2.
\end{equation}
This shows that the sequence of rational approximations $y_n/x_n$ converges to $\beta(2)=0.9159655941\ldots$ as $n\to \infty$, see \eqref{eq2222.100}. This step verifies Theorem \ref{thm9191.300}-i. \\

\textbf{Condition (ii): Monotonically Increasing Ratios $y_n/x_n$}. Utilizing the hypothesis \eqref{eq4577.820}, a pair of consecutive ratios have the forms
\begin{equation}\label{eq4577.850}
\frac{y_{n}}{x_{n}}=\prod_{k \leq n} \frac{p_k^2}{p_k^2-\chi(p_k)}
=\left (\frac{p_n^2}{p_n^2-1}\right)\prod_{k \leq n-1} \frac{p_k^2}{p_k^2-1},
\end{equation}
and 
\begin{equation}\label{eq4577.860}
\frac{y_{n+1}}{x_{n+1}}=\prod_{k \leq n+1} \frac{p_k^2}{p_k^2-\chi(p_k)}=\left (\frac{p_{n+1}^2}{p_{n+1}^2-1}\right)\left (\frac{p_n^2}{p_n^2-1}\right)\prod_{k \leq n-1} \frac{p_k^2}{p_k^2-1}.
\end{equation}
Clearly, this is a monotonically increasing sequence:
\begin{equation}\label{eq4577.870}
\frac{y_{n}}{x_{n}}
=\left (\frac{p_n^2}{p_n^2-1}\right)\prod_{k \leq n-1} \frac{p_k^2}{p_k^2-1}< \frac{y_{n+1}}{x_{n+1}}=\left (\frac{p_{n+1}^2}{p_{n+1}^2-1}\right)\left (\frac{p_n^2}{p_n^2-1}\right)\prod_{k \leq n-1} \frac{p_k^2}{p_k^2-1}.
\end{equation}
This step verifies Theorem \ref{thm9191.300}-ii. \\
 
\textbf{Condition (iii): Monotonically Decreasing Slopes}. A pair of consecutive slopes are computed using the hypothesis \eqref{eq4577.820}. \\ 

The first ratio of shifted differences (slope, see  \cite[Figure 1]{BK1972} for a graphical description of this sequence) is
\begin{equation}\label{eq4577.920}
\frac{y_{n+1}-y_{n}}{x_{n+1}-x_{n}}=\frac{\left (p_{n+1}^2-1 \right ) }{\left ((p_{n+1}^2-1)-1 \right ) }\frac{y_{n}}{x_{n}}=\frac{\left (p_{n+1}^2-1 \right ) }{\left (p_{n+1}^2-2 \right ) }\frac{y_{n}}{x_{n}},
\end{equation}
where $\chi(p_{n+1})=1$. \\

The next ratio of shifted differences is
\begin{eqnarray}\label{eq4577.930}
\frac{x_{n+2}-x_{n+1}}{y_{n+2}-y_{n+1}}&=&\frac{\left (p_{n+2}^2p_{n+1}^2-p_{n+1}^2 \right ) }{\left ((p_{n+2}^2+1)(p_{n+1}^2-1)-(p_{n+1}^2-1) \right ) }\frac{y_{n}}{x_{n}}\\
&=&\frac{\left (p_{n+2}^2-1 \right )p_{n+1}^2 }{\left ((p_{n+2}^2+1)-1 \right )(p_{n+1}^2-1) }\frac{y_{n}}{x_{n}}\nonumber,
\end{eqnarray}
where $\chi(p_{n+2})=-1$.\\

Comparing a pair of consecutive ratios yields
\begin{eqnarray}\label{eq4577.940}
\frac{x_{n+1}-x_{n}}{y_{n+1}-y_{n}}&=&\frac{\left (p_{n+1}^2-1 \right ) }{\left (p_{n+1}^2-2 \right ) }\frac{y_{n}}{x_{n}}\\
&>&\frac{x_{n+2}-x_{n+1}}{y_{n+2}-y_{n+1}}
=\frac{\left (p_{n+2}^2-1 \right )p_{n+1}^2 }{p_{n+2}^2(p_{n+1}^2-1)  }\frac{y_{n}}{x_{n}}\nonumber.
\end{eqnarray}
Equivalently,
\begin{equation}\label{eq4577.950}
\frac{\left (p_{n+1}^2-1 \right ) }{\left (p_{n+1}^2-2 \right ) }>\frac{\left (p_{n+2}^2-1 \right ) p_{n+1}^2}{p_{n+2}^2(p_{n+1}^2-1) }.
\end{equation}
Expanding and simplifying it return
\begin{equation}\label{eq4577.960}
p_{n+1}^4
 >p_{n+1}^2 .
\end{equation}
Therefore, the slope \eqref{eq4577.920} is a strictly monotonically decreasing sequence. This step verifies Theorem \ref{thm9191.300}-iii. \\ 

Therefore, since all the conditions of Theorem \ref{thm9191.300} are satisfied, the number $\beta(2)$ is irrational.
\end{proof}

\begin{rmk}{\normalfont As stated in Section \ref{S9191}, there is a proof for the existence of infinitely many consecutive prime triples of constant congruences $p_n\equiv  p_{n+1}\equiv  p_{n+2}\equiv 1 \bmod 4$ and $p_n\equiv  p_{n+1}\equiv  p_{n+2}\equiv -1 \bmod 4$ on arithmetic progressions of a single variable. However, the same argument fails for any of these sequences of consecutive prime triples. For example, the same analysis as above using the sequence 
\begin{equation}\label{eq4577.600}
 p_n\equiv 1 \bmod 16, \quad p_{n+1}\equiv 5 \bmod 16, \quad \text{ and }\quad p_{n+2}\equiv 9 \bmod 16
\end{equation}
of constant quadratic symbol 
\begin{equation}\label{eq4577.610}
 \chi(p_n)= 1 , \quad \chi(p_{n+1})=1, \quad \text{ and }\quad \chi(p_{n+2})=1,
\end{equation}
or using the sequence 
\begin{equation}\label{eq4577.620}
 p_n\equiv 3 \bmod 16, \quad p_{n+1}\equiv 7 \bmod 16, \quad \text{ and }\quad p_{n+2}\equiv 11 \bmod 16,
\end{equation}
of constant quadratic symbol 
\begin{equation}\label{eq4577.630}
 \chi(p_n)= -1 , \quad \chi(p_{n+1})=-1, \quad \text{ and }\quad \chi(p_{n+2})=-1,
\end{equation}
fails to prove that $\beta(2)$ is irrational.
}
\end{rmk}

\section{Basic Diophantine Approximations } \label{S3300T}
The basic results recorded below are standard results in the literature, see \cite{RD1996}, \cite{RH1994}, et alii. 

\subsection{Basic Continued Fractions} \label{S3300TA}
\begin{lem} \label{lem3300T.100} Let $\alpha=\left [ a_0, a_1, \ldots, a_n, \ldots, \right ]$ be the continue fraction of the real number $\alpha \in \R$. Then, the following properties hold.
\begin{enumerate} [font=\normalfont, label=(\roman*)]
\item$ \displaystyle  p_n=a_np_{n-1}+p_{n-2},$ \tabto{6cm} $p_{-2}=0, \quad p_{-1}=1$, \quad for all $n\geq 0.$
\item$ \displaystyle  q_n=a_nq_{n-1}+q_{n-2},$ \tabto{6cm} $q_{-2}=1, \quad q_{-1}=0$, \quad for all $n\geq 0.$
\item$ \displaystyle  p_nq_{n-1}-p_{n-1}q_{n}=(-1)^{n-1},$ \tabto{6cm} for all $n\geq 1.$
\item$ \displaystyle  \frac{p_n}{q_{n}}=a_0+\sum_{0 \leq k < n}\frac{(-1)^{k}}{q_kq_{k+1}},$ \tabto{6cm} for all $n\geq 1.$

\end{enumerate}
\end{lem}

\subsection{The Irrationality Measure } \label{S3300TB}
The irrationality measure measures the quality of the rational approximation of an irrational number. 
It is lower bound of all the rational approximations. Specifically,
\begin{equation}\label{eq3300T.130}
 \left |\alpha -\frac{p}{q} \right |\gg  \frac{1}{q^{\mu}}
\end{equation}
for all sufficiently large $q\ge1$.

\section{Numerical Data} \label{S3300N}
The data to support the claim in Conjecture \ref{conj2222.300} is compiled in this section. The continued fraction of the Catalan constant 
\begin{equation}\label{eq3300N.200}
\beta(2)=0.91596559417721901505460351493238411077414937428167\ldots,
\end{equation}
is of the form, (first 100 partial quotients $a_n$),
\begin{eqnarray}\label{eq3300N.210}
\beta(2)&=&[0; 1, 10, 1, 8, 1, 88, 4, 1, 1, 7, 22, 1, 2, 3, 26, 1, 11, 1, 10, 1, 9, 3, 1, 1, 1, 1,\nonumber\\&& 1, 1, 2, 2, 1, 11, 1, 1, 1, 6, 1, 12, 1, 4, 7, 1, 1, 2, 5, 1, 5, 9, 1, 1, 1, 1, 33, 4, 1, \nonumber\\&& 1, 3, 5, 3, 2, 1, 2, 1, 2, 1, 7, 6, 3, 1, 3, 3, 1, 1, 2, 1, 14, 1, 4, 4, 1, 2, 4, 1, 17, 4, \nonumber\\&&1, 14, 1, 1, 1, 12, 1, 1, 1, 3, 1, 2, 3, 1, 6, 2, 1, 2, 2, 322, 1, 1, 1, 2, 1, 108, 3, 1,\nonumber\\&& 2, 82, 1, 5, 4, 1, 2, 2, 1, 1, 1, 5, 1, 12, 2, 11, 8, 2, 17, 1, 11, 1, 6, 1, 18, 1, 5, 2, \nonumber\\&&24, 4, 1, 1, 1, 8, 4, 3, 8, 3, \ldots].  
\end{eqnarray}
These are archived as sequence A006752 and sequence A014538, respectively, on the OEIS.\\

The sequence of convergents $\{p_n/q_n: n \geq 1\}$, listed in Table \ref{t2000.500}, is computed via the recursive formula provided in the Lemma \ref{lem3300T.100}. \\

An approximation $\mu_n(\alpha)$ of the irrationality measure satisfies the inequality
\begin{equation} \label{eq3300N.230}
  \left | \alpha-\frac{p_n}{q_n} \right | 
\geq\frac{1}{q^{\mu_n(\alpha) }}
\end{equation}
for $n\geq 2$. The values of the approximate irrationality measure $\mu_n(\alpha)\geq2$ of the irrational number $\alpha\ne0$ is defined by
\begin{equation}\label{eq5555N.500}
\mu_n(\alpha)=-\frac{\log \left |\alpha-p_n/q_n\right | }{\log q_n}, 
\end{equation}
where $n\geq 2$. 

\begin{exa}{\normalfont A large convergent is used here to illustrate the calculations, using 50 digits accuracy in the computer algebra system SAGE. The 100th convergent $p_n/q_n$ is given by
\begin{enumerate}
\item[(a)] $ \displaystyle p_{100}=
24078868662746347429760476964387436156348637833,$
\item[(b)] $ \displaystyle q_{100}=26287961923259336649196821919541159881600485419
.$
\end{enumerate} 
The corresponding $100$th approximation of the irrationality measure is
\begin{eqnarray}\label{eq5555N.510}
\mu_{100}(\beta(2))&=&-\frac{\log \left |\beta(2)-p_{100}/q_{100}\right | }{\log q_{100}}\\&=&2.0098375679109850809407389673548425452383093096668\nonumber. 
\end{eqnarray} 
}
\end{exa} 
The range of values for $n\leq 45$ is plotted in Figure \ref{fi3300N.500}. 
\begin{figure}[h]
\caption{Approximate Irrationality Measure $\mu_n(\beta(2))$ Of The Number $\beta(2).$} \label{fi3300N.500}\centering%
  \begin{tikzpicture}
	\begin{axis}[
		xlabel=$n$,
		ylabel=$\mu_n(\beta(2))$,
width=0.95\textwidth,
       height=0.5\textwidth		]
	\addplot[color=blue,mark=square] coordinates {
(2,2.07678354372)
(3,2.92280567895)
(4,2.02491750642)
(5,2.93949885584)
(6,2.16337045430)
(7,2.07073791493)
(8,2.06103252758)
(9,2.17687224408)
(10,2.23191137105)
(11,2.02341311597)
(12,2.07153881185)
(13,2.06867467056)
(14,2.17463739738)
(15,2.00518561271)
(16,2.11498233047)
(17,2.00653757661)
(18,2.09965407158)
(19,2.00668821796)
(20,2.08514590463)
(21,2.04442502885)
(22,2.02162846316)
(23,2.02785664628)
(24,2.02499038253)
(25,2.02500515325)
(26,2.02579968126)
(27,2.02174828551)
(28,2.03244636822)
(29,2.03473350242)
(30,2.01160242732)
(31,2.07058476506)
(32,2.01265439798)
(33,2.02693536878)
(34,2.01313249833)
(35,2.05164524679)
(36,2.00499959412)
(37,2.06337796886)
(38,2.00633061958)
(39,2.03693255654)
(40,2.04514224382)
(41,2.01265617595)
(42,2.01784883868)
(43,2.02062399817)
(44,2.03721810851)
(45,2.00635311938)
		
	};
	\addplot[color=black,mark=square] coordinates {
		(2,2)
		(3,2)
		(4,2)
		(5,2)
		(6,2)
		(7,2)
		(8,2)
		(9,2)
		(10,2)
		(11,2)
		(12,2)
		(13,2)
		(14,2)
		(15,2)
		(16,2)
		(17,2)
		(18,2)
		(19,2)
		(20,2)
		(21,2)
		(22,2)
		(23,2)
		(24,2)
		(25,2)
		(26,2)
		(27,2)
		(28,2)
		(29,2)
		(30,2)
		(31,2)
		(32,2)
		(33,2)
		(34,2)
		(35,2)
		(36,2)
		(37,2)
		(38,2)
		(39,2)
		(40,2)
		(41,2)
		(42,2)
		(43,2)
		(44,2)
		(45,2)
	};
	\end{axis}
\end{tikzpicture}	
\end{figure}
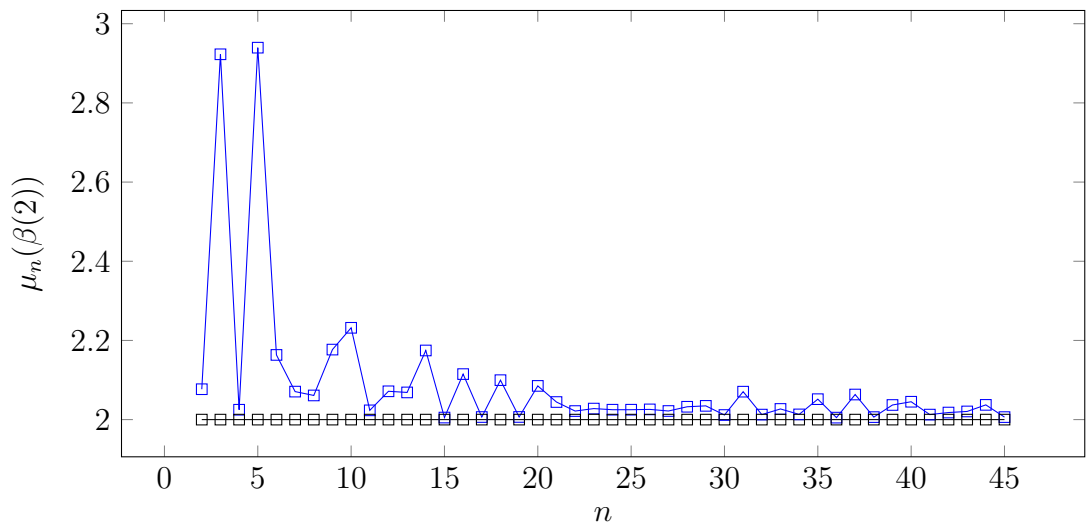


\begin{table}[h!]
\centering
\caption{Numerical Data For The Exponent $\mu(\beta(2))$ Of The Number $\beta(2)$} \label{t2000.500}
\begin{tabular}{l|l|l| l}
$n$&$p_n$&$q_n$&$\mu_n(\beta(2))$\\
\hline
1&       1&        1       &    \\   
  2&       10&        11       &2.07678354372  \\     
  3&       11&        12       &2.92280567895    \\   
  4&       98&        107       &2.02491750642     \\  
  5&       109&        119       &2.93949885584      \\ 
  6&       9690&        10579       & 2.16337045430    \\  
  7&       38869&        42435       &2.07073791493      \\ 
  8&       48559&        53014       & 2.06103252758      \\
  9&       87428&        95449       &2.17687224408       \\
  10&       660555&        721157       & 2.23191137105     \\ 
  11&       14619638&        15960903       &2.02341311597    \\   
  12&       15280193&        16682060       & 2.07153881185     \\ 
  13&       45180024&        49325023       & 2.06867467056      \\
  14&       150820265&        164657129       & 2.17463739738      \\
  15&       3966506914&        4330410377       &  2.00518561271     \\
  16&       4117327179&        4495067506       & 2.11498233047      \\
  17&       49257105883&        53776152943       & 2.00653757661      \\
  18&       53374433062&        58271220449       &2.09965407158       \\
  19&       583001436503&        636488357433       &2.00668821796       \\
  20&       636375869565&        694759577882       & 2.08514590463      \\
  21&       6310384262588&        6889324558371       &2.04442502885       \\
  22&       19567528657329&        21362733252995       &  2.02162846316     \\
  23&       25877912919917&        28252057811366       & 2.02785664628      \\
  24&       45445441577246&        49614791064361       & 2.02499038253      \\
  25&       71323354497163&        77866848875727       & 2.02500515325      \\
  26&       116768796074409&        127481639940088       & 2.02579968126    \\
  27&       188092150571572&        205348488815815       & 2.02174828551    \\
  28&       304860946645981&        332830128755903       &2.03244636822     \\
  29&       797814043863534&        871008746327621       &2.03473350242    \\
  30&       1900489034373049&        2074847621411145       & 2.01160242732  \\
  31&       2698303078236583&        2945856367738766       &2.07058476506  \\
  32&       31581822894975462&        34479267666537571   & 2.01265439798   \\
  33&       34280125973212045&        37425124034276337     & 2.02693536878 \\
  34&       65861948868187507&   71904391700813908    &2.01313249833       \\
  35&       100142074841399552&        109329515735090245   &2.05164524679  \\
  36&       666714397916584819&        727881486111355378    &2.00499959412\\
  37&       766856472757984371&        837211001846445623   &2.06337796886 \\
  38&       9868992071012397271&     10774413508268702854   &2.00633061958\\ 
  39&       10635848543770381642&    11611624510115148477  &2.03693255654  \\
  40&       52412386246093923839&  57220911548729296762    &2.04514224382   \\
  41&       377522552266427848515&   412158005351220225811  &2.01265617595 \\
  42&       429934938512521772354&  469378916899949522573   &2.01784883868 \\
  43&       807457490778949620869&  881536922251169748384   &2.02062399817 \\
  44&       2044849920070421014092&  2232452761402289019341 & 2.03721810851\\
  45&       11031707091131054691329& 12043800729262614845089 & 2.00635311938\\ 
\end{tabular}
\end{table}

\currfilename.\\

\end{document}